\documentclass[12pt]{amsart}
\usepackage{latexsym}
\usepackage{amssymb}
\usepackage[pdftex]{hyperref}

\newtheorem{theorem}{Theorem}[section]
\newtheorem{lemma}[theorem]{Lemma}
\newtheorem{proposition}[theorem]{Proposition}

\theoremstyle{definition}

\theoremstyle{remark}
\newtheorem{remark}[theorem]{Remark}

\numberwithin{equation}{section}

\newcommand{\SL}{{\mathrm {SL}}}
\newcommand{\PSL}{{\mathrm {PSL}}}

\newcommand{\SU}{{\mathrm {SU}}}
\newcommand{\PSU}{{\mathrm {PSU}}}

\newcommand{\Spin}{{\mathrm {Spin}}}

\newcommand{\Sp}{{\mathrm {Sp}}}
\newcommand{\PSp}{{\mathrm {PSp}}}

\newcommand{\Aut}{{\mathrm {Aut}}}

\newcommand{\Mult}{{\mathrm {Mult}}}
\newcommand{\Schur}{{\mathrm {Schur}}}

\newcommand{\Irr}{{\mathrm {Irr}}}

\newcommand{\St}{{\mathrm {St}}}

\newcommand{\CC}{{\mathbb C}}

\newcommand{\ZZ}{{\mathbb Z}}
\newcommand{\NN}{{\mathbb N}}

\newcommand{\ta}{\hspace{0.5mm}^{2}\hspace*{-0.2mm}}

\newcommand{\cd}{{\mathrm {cd}}}

\begin{document}
\title[Complex group algebras of quasisimple classical groups]
{Quasisimple classical groups\\ and their complex group algebras}

\author{Hung Ngoc Nguyen}
\address{Department of Mathematics, The University of Akron, Akron,
Ohio 44325, United States} \email{hungnguyen@uakron.edu}



\date{\today}

\begin{abstract} Let $H$ be a finite quasisimple classical group, i.e. $H$ is perfect and $S:=H/Z(H)$ is a finite simple classical
group. We prove in this paper that, excluding the cases when the
simple group $S$ has a very exceptional Schur multiplier such as
$\PSL_3(4)$ or $\PSU_4(3)$, $H$ is uniquely determined by the
structure of its complex group algebra. The proofs make essential
use of the classification of finite simple groups as well as the
results on prime power character degrees and relatively small
character degrees of quasisimple classical groups.
\end{abstract}

\maketitle

%%% ------------------------------------------------------------------------------------------------------------------------

\section{Introduction}\label{section Introduction}

An important question in representation theory of finite groups is
the extent to which the complex group algebra of a finite group
determines the group or its properties. In late 1980s, Isaacs proved
that if $\CC G\cong\CC H$ and $p$ is a prime, then $G$ has a normal
$p$-complement if and only if $H$ has a normal $p$-complement, and
therefore the nilpotency of a group is determined by the complex
group algebra of the group (cf.~\cite{Isaacs}). Later on, Hawkes
gave a counterexample showing that the same statement does not hold
for supersolvability (cf.~\cite{Hawkes}). It is still unknown
nowadays whether the solvability of a finite group is preserved by
its complex group algebra
(cf.~\cite[Problem~11.8]{Mazurov-Khukhro}).

While some properties of a solvable group $G$ is determined by $\CC
G$, it is well known that $G$ in general is not determined by $\CC
G$. For instances, $\CC D_8\cong \CC Q_8$ or Dade even constructed
two non-isomorphic metabelian groups with isomorphic group algebras
over any field (cf.~\cite{Dade}).

In contrast to solvable groups, simple groups or more generally
quasisimple groups are believed to have a stronger connection with
their complex group algebras. Recently, Tong-Viet in a series of
papers~\cite{Tong-Viet1,Tong-Viet2,Tong-Viet} has succeeded in
proving that, if $G$ is a finite group and $H$ is a finite simple
group such that $\CC G\cong\CC H$, then $G\cong H$.

The main purpose of this paper is to go one step further and improve
Tong-Viet's result to quasisimple groups in the case of classical
groups.

\begin{theorem}\label{main theorem} Let $G$ be a finite group and let $H$ be a quasisimple group such that $H/Z(H)$ is
a simple classical group different from $\PSL_3(4)$ and $\PSU_4(3)$.
If $\CC G\cong\CC H$ then $G\cong H$. In other words, every
quasisimple classical group $H$ with $H/Z(H)\neq\PSL_3(4)$ and
$\PSU_4(3)$ is uniquely determined (up to isomorphism) by the
structure of its complex group algebra.
\end{theorem}

We should mention that, as the complex group algebra of a finite
group determines and also is determined by the degrees (counting
multiplicities) of irreducible characters of the group, the
statement of the theorem is equivalent to say that, excluding those
exceptions, every quasisimple classical group is determined by its
\emph{multiset of character degrees} or its
\emph{character-degree-frequency function} (cf.~\cite{Hawkes} for
the definition).

If $G$ is a finite group and $H$ is a quasisimple group having the
same complex group algebra then it is easy to see that they have the
same order and multiset of character degrees (i.e., the set of
character degrees counting multiplicities). Moreover, as $H$ has a
unique linear character, so does $G$ and hence $G$ is perfect. We
will prove that, if $M$ is a maximal normal subgroup of $G$ then
$G/M$ must be isomorphic to $S:=H/Z(H)$. This basically eliminates
the involvement of all nonabelian simple group other than $S$ in the
structure of $G$. Let $\cd(G)$ denote the character degree set of
$G$ and let $\Schur(S)$ denote the Schur cover of $S$. On the way to
the proof of Theorem~\ref{main theorem}, we in fact prove the
following:

\begin{proposition}\label{non-containment of character degrees of quasisimple groups} Let $S$ be a simple classical group. Let $G$ be a nontrivial perfect group and $M$
be a maximal subgroup of $G$ such that $|S|\mid|G|\mid |\Schur(S)|$
and $\cd(S)\subseteq\cd(G)\subseteq \cd(\Schur(S))$. Then $G/M\cong
S$.
\end{proposition}

This proposition will be proved in sections~\ref{section-linear
groups}, \ref{section-unitary groups}, \ref{section-symplectic
groups}, and~\ref{section-orthogonal groups in even dimension} for
linear groups, unitary groups, symplectic groups, and orthogonal
groups, respectively. Assume for now that it is true. Since
$|G|=|H|$ and $G/M\cong H/Z(H)$, we deduce that $|M|=|Z(H)|$. It
follows that, if $H$ is simple then $M$ is trivial and we would have
immediately that $G\cong H$. However, since we are working on a
quasisimple group $H$, the problem becomes much more complicated. To
illustrate the difficulty of the problem for quasisimple groups, let
us mention that we are unable to answer whether the complex group
algebra of a quasisimple group, whose quotient by its center is
$\PSL_3(4)$ or $\PSU_4(3)$, determines the group uniquely up to
isomorphism. This is due to the very exceptional Schur multiplier of
these groups (cf.~Lemmas~\ref{S cannot be embedded in Aut(A)}
and~\ref{G/Mi is isomorphic to a quotient of Schur(S)}).

In section~\ref{section-main result}, we will show that if $H/Z(H)$
is different from $\PSL_3(4)$ and $\PSU_4(3)$, then $H/Z(H)$ can not
be imbedded into the automorphism group of an abelian group of order
not larger than that of the Schur multiplier of $H/Z(H)$. We then
use this to show that $G$ is isomorphic to a quotient of the Schur
cover of $H/Z(H)$ (cf.~Lemma~\ref{G is isomorphic to a quotient of
Schur(S)}). From there, with a careful analysis of the Schur
multipliers of finite simple groups, we are able to show that
$G\cong H$. We note that many results in section~\ref{section-main
result} are stated for every finite group of Lie type, not only
classical groups.

With the techniques presented here, we hope to continue the problem
for the Schur covers of symmetric and alternating groups
(cf.~\cite{Nguyen-TongViet}), quasisimple exceptional groups of Lie
type as well as quasisimple sporadic groups. To end this
introduction, we put a list of notation that will be used throughout
the paper.

\medskip

\begin{tabular}{lll}
Notation & Meaning\\\hline
$Z(G)$& the center $G$\\
$\Aut(G)$& the automorphism group of $G$\\
$\Irr(G)$& the set of irreducible characters of $G$\\
$\cd(G)$ & the set of irreducible character degrees of $G$\\
$d_i(G)$& the $i$th smallest number in $\cd(G)\backslash\{1\}$\\
$G^{(i)}$& the $i$th derived subgroup of $G$\\
 $\Mult(G)$ & the Schur multiplier of $G$\\
$\Schur(G)$& the Schur cover of a perfect group $G$ \\
 $\St_G$& the Steinberg character of a group of Lie type $G$\\
 $n_p$, $p$ prime & the $p$-part of an integer number $n$\\\hline
\end{tabular}

%%%%%%%%%%%%%%%%%%%%%%%%%%%%%%%%%%%%%%%%%%%%%%%%%%%%%%%%%%%%%%%%%%%%%%%%%%%%%%%%%%%%%%%%%%%%%%%%%%

\section{Character degrees of quasisimple classical groups}\label{section-character
degrees}

In this section, we prove a part of Proposition~\ref{non-containment
of character degrees of quasisimple groups}. More explicitly, we
show that $G/M$ is a simple group of Lie type in characteristic $p$.
We start with a couple of known lemmas.

\begin{lemma}[Gallagher's lemma]\label{Gallagher} Let $M\unlhd G$ and $\chi\in \Irr(G)$. If
$\chi_M\in\Irr(M)$ then $\chi\theta\in\Irr(G)$ for every
$\theta\in\Irr(G/M)$.
\end{lemma}

\begin{lemma}[\cite{Huppert1}, Lemma~3]\label{lemmaHuppert2} Let $M\unlhd
G$ and let $\theta\in\Irr(M)$ be $G$-invariant. If $\chi\in\Irr(G)$
lying above $\theta$, then $\chi=\theta_0\tau$, where $\theta_0$ is
a character of an irreducible projective representation of $G$ of
degree $\theta(1)$ and $\tau$ is a character of an irreducible
projective representation of $G/M$.
\end{lemma}

\begin{lemma}[\cite{Tong-Viet1}, Lemma~2.2]\label{maximal degree of Am is
2^{m-1}} If $n\geq10$ then the maximal degree of $A_n$ is at least
$2^{n-1}$.
\end{lemma}

The next lemma is a special property of the Steinberg character,
denoted by $\St_S$, of a simple group of Lie type $S$ in
characteristic $p$. It is well known that the degree of $\St_S$ is
$|S|_p$.

\begin{lemma}\label{degree of S whose no proper multiple is degree of
Schur(S)} Let $S$ be simple group of Lie type. Then no proper
multiple of $\St_S(1)$ is a degree of $\Schur(S)$.
\end{lemma}

\begin{proof} If $S$ is one of the following groups of exceptional
Schur multiplier:
$$\PSL_2(4),\PSL_2(9),\PSL_3(2),\PSL_3(4),\PSL_4(2),\Omega_7(2),\Omega_7(3),P\Omega_8^+(2),$$$$\PSU_4(2),\PSU_4(3),\PSU_6(2),F_4(2),G_2(3),G_2(4)\ta
E_6(2),Suz(8),$$ the statement can be checked directly
using~\cite{Atl1}. So we assume that $S$ is none of these. Then
$\Schur(S)$ is indeed the finite Lie-type group of simply-connected
type corresponding to $S$ and moreover $|S|_p=|\Schur(S)|_p$.
Therefore, using \cite{Curtis}, we have that $\St_S$ is the unique
character of $p$-defect zero of $\Schur(S)$ where $p$ is the
characteristic of the underlying field of $S$.
\end{proof}

The next result is a slight improvement of a result of Tong-Viet
in~\cite[Proposition~3.1]{Tong-Viet}, which has a stronger
hypothesis that $|G|=|S|$ and $\cd(G)=\cd(S)$.

\begin{proposition}\label{G/M is a simple group of Lie type in characteristic
p} Let $S$ be a simple classical group defined over a field of
characteristic $p$. Let $G$ be a nontrivial perfect group and $M$ be
a maximal subgroup of $G$ such that $|S|\mid|G|\mid |\Schur(S)|$ and
$\cd(S)\subseteq\cd(G)\subseteq \cd(\Schur(S))$. Then $L:=G/M$ is a
simple group of Lie type in characteristic $p$.
\end{proposition}

\begin{proof} The proof can be adapted from that
of~\cite[Proposition~3.1]{Tong-Viet} with some minor modifications.
Therefore, we only sketch the main ideas of the proof and leave some
detailed check to the reader. From the hypothesis, we have $|L| \mid
|\Schur(S)|$, $\pi(L)\subseteq \pi(\Schur(S))$, and $d_i(L)\geq
d_i(\Schur(S))$ for every $i$. As $\cd(S)\subseteq\cd(G)$ and
$\St_S(1)=|S|_p$, $G$ has an irreducible character, say $\chi$, of
degree $|S|_p$. Let $\theta\in\Irr(M)$ lying under $\chi$ and
$I:=I_G(\theta)$.

\medskip

(i) Assume that $L$ is a sporadic group or the Tits group. If
$\theta$ is not $G$-invariant, then $I$ would be a proper subgroup
of $G$. By Clifford theory, $\chi=\phi^G$ for some
$\phi\in\Irr(I|\theta)$ and also $|S|_p=\chi(1)=|G:I|\phi(1)$. Using
the classification of subgroups of prime power index in a simple
group by Guralnick (cf.~\cite{Guralnick}), we have that $$L=M_{11},
I/M=M_{10} \text{ and } |G:I|=11,$$ or $$L=M_{23}, I/M=M_{22} \text{
and } |G:I|=23.$$ In particular, the proper subgroup $I/M$ of $G/M$
has order coprime to $p$. It then follows that $(\phi(1),|I:M|)=1$
and hence $\theta$ extends to $\phi$. Together with Gallagher's
lemma, we obtain $\phi\psi\in\Irr(I|\theta)$ for every
$\phi\in\Irr(I/M)$. Therefore $(\phi\psi)^G$ is an irreducible
character of $G$ of degree $\chi(1)\psi(1)=|S|_p\psi(1)$. Taking
$\psi$ to be a nonlinear character of $I/M$, we get a contradiction
by Lemma~\ref{degree of S whose no proper multiple is degree of
Schur(S)}.

So $\theta$ is $G$-invariant and hence $\chi_M=e\theta$ for some
$e\in\NN$. If $e=1$, $\chi$ would be an extension of $\theta$ and
hence $\chi\psi\in\Irr(G)$ for every $\psi\in\Irr(L)$. Taking $\psi$
to be nonlinear and note that $\cd(G)\subseteq\cd(\Schur(S))$, we
deduce that $\Schur(S)$ has a degree which is a proper multiple of
$\chi(1)=|S|_p$, violating Lemma~\ref{degree of S whose no proper
multiple is degree of Schur(S)} again. Thus $e$ is a nontrivial
$p$-power. Moreover, by Lemma~\ref{lemmaHuppert2}, $e$ is a degree
of a projective irreducible representation of $L$. By the
classification of prime power degrees of irreducible characters of
quasisimple groups by Malle and Zalesskii
(cf.~\cite{Malle-Zalesskii}), we come up with the following cases:
$$(L,e)=(M_{11},11),(M_{11},2^4),(M_{12},11),(M_{12},2^4),(M_{12},2^5),(M_{24},23),$$$$(Co_2,23),(Co_3,23),(J_2,2^6),(Ru,2^{13}),(\ta F_4(2)',3^3),\text{ or } (\ta F_4(2)',2^{11}).$$

The cases $L=M_{11}$ and $M_{12}$ can not happen since no
quasisimple classical groups in characteristic $2$ or $11$ have both
degrees $11$ and $2^4$ by ~\cite[Theorem~1.1]{Malle-Zalesskii}.
Similarly, $L\neq \ta F_4(2)'$ since no quasisimple classical groups
in characteristic $2$ or $3$ have both degrees $3^3$ and $2^{11}$.
If $(L,e)=(M_{24},23),(Co_2,23)$, or $(Co_3,23)$ then $p=23$ and
moreover $23$ must be the degree of the Steinberg character of $S$.
However, one can check from~\cite{Atl1} that, in these cases, $L$
has a degree $253$, which is a proper multiple of $23$ and this
violates Lemma~\ref{degree of S whose no proper multiple is degree
of Schur(S)}. For the remaining cases $(L,e)=(J_2,2^6)$ or
$(Ru,2^{13})$, we can argue as follows. Suppose that $S=S_n(q)$
where $n$ is the dimension of $S$ and $q$ is the cardinality of the
underlying field of $S$. Since $d_1(L)\geq d_1(\Schur(S))$, by using
the result on minimal degrees of quasisimple classical groups of
Tiep and Zalesskii in~\cite{TZ1}, one can bound $n$ and $q$ and come
up with a short list of possibilities for $S$. At this stage, one
just checks directly by using~\cite{Atl1} or~\cite{Lubeck}
or~\cite{GAP4} to verify that $\cd(L)\nsubseteq\cd(\Schur(S))$.

\medskip

(ii) Assume that $L$ is an alternating group $A_m$ with $m=7$ or
$m\geq9.$ We note that $A_5,A_6$, and $A_8$ can be considered as
Lie-type groups. Since the arguments for each family of classical
groups are fairly similar, we only present here the linear group
case.

Suppose that $S=\PSL_n(q)$. Using~\cite{Atl1}, one can easily
eliminate those groups with exceptional Schur multiplier such as:
$\PSL_2(4)=\PSL_2(5)$, $\PSL_2(9)$, $\PSL_3(2)$, $\PSL_3(4)$, and
$\PSL_4(2)$ and also the group $\PSL_2(7)$. So we assume $S$ is none
of these and hence $\Schur(S)=\SL_n(q)$. Assume first that $m\leq9$.
Then we have $d_1(\SL_n(q))\leq d_1(A_m)\leq8$. Using the table of
minimal degrees of $\SL_n(q)$ in~\cite[Table~IV]{TZ1} again, we
deduce that $(n,q)=(2,q \text{ odd }\leq17)$ or $(2,q \text{ even
}\leq 8)$. It is now routine to check that $\cd(A_m)\nsubseteq
\cd(\SL_n(q))$ for these cases.

Thus we can assume that $m\geq10$. It then follows by the hypothesis
and Lemma~\ref{maximal degree of Am is 2^{m-1}} that
$$b(\SL_n(q))\geq b(L)=b(A_m)\geq 2^{m-1}\geq 2^{d_1(L)}\geq 2^{d_1(\SL_n(q))}.$$
With the list of character degrees of Lie-type groups of low rank
available in~\cite{Lubeck}, one can check that this inequality is
violated for $n=2,3,4$. For $n\geq5$, as
$d_1(\SL_n(q))=(q^n-q)/(q-1)$ by~\cite[Table~IV]{TZ1} and
$b(\SL_n(q))<\sqrt{|\SL_n(q)|}<q^{n^2/2}$, the above inequality
implies $$2^{(q^n-q)/(q-1)}< q^{n^2/2},$$ which in turn implies
$$2q^{n-1}<\log_2q\cdot n^2.$$ This last inequality is impossible
when $n\geq5$.

\medskip

(iii) Assume that $L$ is a simple group of Lie type in
characteristic $r\neq p$. Then, as $\cd(L)\in\cd(\Schur(S))$,
$\cd(\Schur(S))$ has at least two prime power members: $|S|_p$ and
$|L|_r$. Again using ~\cite{Malle-Zalesskii}, we get a list of
possibilities for $S$ and $|L|_r$. These possibilities can be
eliminated similarly as in~\cite[Proposition~3.1]{Tong-Viet} except
the following one: $$S=\PSL_2(q), |L|_r=q\pm1 \text{ or }
|L|_r=(q\pm1)/2 \text{ with odd } q.$$ In fact, Tong-Viet assumed in
his proposition that the simple classical group $S$ has rank $\geq3$
and therefore $\PSL_2(q)$ was not in his consideration. We now
present the arguments to eliminate this remaining possibility.
Again, as it is routine to check the lemma for small $q$, we assume
that $q\geq11$. We then have that $\Schur(S)=\SL_2(q)$ and therefore
$\cd(L)\subseteq \cd(\SL_2(q))$. Since $\SL_2(q) (q\geq11)$ has at
most $5$ character degrees, we deduce that $L=\SL_2(q_1)$ for some
$q_1$ coprime to $p$. So we have $$\{q_1,q_1-1,q_1+1\}\subseteq
\cd(L)\subseteq \cd(\SL_2(q))\subseteq\{q,q-1,q+1,(q-1)/2,
(q+1/2)\},$$ which is impossible as $q\geq 11$.
\end{proof}

%%%%%%%%%%%%%%%%%%%%%%%%%%%%%%%%%%%%%%%%%%%%%%%%%%%%%%%%%%%%%

\section{Linear groups}\label{section-linear groups}

We have seen from Proposition~\ref{G/M is a simple group of Lie type
in characteristic p} that $L:=G/M$ is a simple group of Lie type in
characteristic $p$ and so we write $L=L(p^b)$ to indicate that $L$
is defined over a field of $p^b$ elements. Assuming the hypothesis
of Proposition~\ref{non-containment of character degrees of
quasisimple groups}, we have $$|L|\mid |\Schur(S)| \text{ and }
\cd(L)\subseteq \cd(\Schur(S)).$$ From now on to the end of
section~\ref{section-orthogonal groups in even dimension}, we often
use these conditions without notice.

The following classical result of Zsigmondy is very useful in
proving the non-divisibility between orders and also character
degrees of classical groups.

\begin{lemma}[Zsigmondy's theorem] If $x>y > 0$ are coprime integers, then for any natural number $n > 1$
there is a prime number denoted by $\ell(x,y,n)$ (called a primitive
prime divisor) that divides $x^n -y^n$ and does not divide $x^i -
y^i$ for any positive integer $i < n$, with the following
exceptions:
\begin{enumerate}
\item[(i)] $x = 2, y = 1$, and $n = 6$; or
\item[(ii)] $x + y$ is a power of two, and $n = 2$.
\end{enumerate}
%Similarly, $x^n + y^n$ has at least one primitive prime divisor with
%the exception $2^3 + 1^3 = 9$.
\end{lemma}

\begin{table}[h]\caption{Possible nontrivial character degrees up to $q^n$ of $\SL_n(q)$ (cf.~\cite{Lubeck,TZ1}).}\label{table-linear}
\begin{tabular}{ll}
Value of $n$ & Degrees\\\hline
$n=2$ & \begin{tabular}{l}$(q\pm1)/2,q,q\pm1$\end{tabular}\\
 $n=3$ &\begin{tabular}{l}$q(q+1),q^2+q+1,(q+1)(q-1)^2,q^3-1,$\\$(q+1)(q-1)^2/3,(q+1)(q^2+q+1)/3$\end{tabular}\\
$n=4$&\begin{tabular}{l}$q(q^2+q+1),(q+1)(q^2+1),(q-1)(q^3-1),$\\$(q-1)(q^3-1)/2,(q^2+1)(q^2+q+1)/2$\end{tabular}\\
 $n\geq5$& \begin{tabular}{l}$(q^n-q)/(q-1), (q^n-1)/(q-1)$\end{tabular}\\\hline
\end{tabular}
\end{table}

\begin{lemma}\label{lemma for linear groups} Proposition~\ref{non-containment of character degrees of quasisimple
groups} is true when $S$ is a simple linear group.
\end{lemma}

\begin{proof} One can verify the lemma easily for the following groups: $S=\PSL_2(4)\cong\PSL_2(5)$, $\PSL_2(9)$, $\PSL_3(2)$, $\PSL_3(4)$, and $\PSL_4(2)$
by using~\cite{Atl1}. So we assume that $S=\PSL_n(q)=\PSL_n(p^a)$ is
none of these. In that case, $\Schur(S)=\SL_n(p^a)$ and
$\St_S(1)=p^{an(n-1)/2}$ is the only $p$-power degree of $\SL_n(q)$
by the classification of prime power degrees of quasisimple groups
in~\cite[Theorem~1.1]{Malle-Zalesskii}. Recall from
Proposition~\ref{G/M is a simple group of Lie type in characteristic
p} that $L:=G/M$ is a simple group of Lie type in characteristic $p$
and we aim to show that $L\cong S$.

\medskip

(i) We first outline the arguments to eliminate the simple Lie-type
groups of exceptional type. Suppose that $|L|_p=p^{bm}$ and the
largest factor of the form $p^x-1$ in the formula of $|L|$ is
$p^{bm'}-1$. Then, as $\St_L(1)=p^{bm}$ is a $p$-power degree of
$\SL_n(q)$, we get $p^{bm}=p^{an(n-1)/2}$ and therefore
\begin{equation}\label{bm=an(n-1)/2}bm=an(n-1)/2.\end{equation} It follows
that
$$bm'=\frac{an(n-1)}{2}\cdot \frac{m'}{m}.$$
As $$(p^{bm'}-1)\mid |L|\mid |\SL_n(q)|=p^{an(n-1)/2}\prod_{i=2}^n
(p^{ia}-1),$$ using Zsigmondy's theorem, one gets
$$\frac{an(n-1)}{2}\cdot \frac{m'}{m}\leq an \text{ or } \frac{an(n-1)}{2}\cdot
\frac{m'}{m}=6>an.$$ Therefore,
$$n\leq \max\{5,2m/m'+1\}.$$ Since $m/m'$ is known (the maximum value of $m/m'$ is $4$, obtained when
$L=E_8$), we can bound above the value of $n$ ($n$ indeed is at most
$4\cdot 2+1=9$ for all groups of exceptional type). For small $n$,
we use Equation~\ref{bm=an(n-1)/2} to show that either $|L|\nmid
|\SL_n(q)|$ or $\cd(L)\nsubseteq \cd(\SL_n(q))$ with the help
of~\cite{Lubeck} and Zsigmondy's theorem. To illustrate this, let us
present the arguments for the most difficult case $L(p^b)=E_8(p^b)$.

Assume by contrary that $L(p^b)=E_8(p^b)$. Then $m=120$ and $m'=30$
by~\cite[p.~xvi]{Atl1}. Therefore $$n\leq \max\{5,2\cdot
120/30+1\}=9.$$ For every $n\leq 9$, we in fact obtain a
contradiction by showing that $|E_8(p^b)|\nmid |\SL_9(p^a)|$. For
instance, if $n=9$ then $120b=36a$ by~\eqref{bm=an(n-1)/2}.
Therefore $b=9c$ and $a=30c$ for some positive integer $c$. Now the
Zsigmondy prime $\ell(p,1,216c)$ is a divisor of $|E_8(p^{9c})|$ but
does not divide $|\SL_9(p^{30c})|$. Other values of $n\leq8$ are
handled similarly.

\medskip

(ii) Next, we eliminate the remaining simple classical groups in
characteristic $p$ except $\PSL_n(p^a)$.

$\bullet$ $L=\Omega_{2m+1}(p^b)$ or $\PSp_{2m}(p^b)$ with $m\geq2$.
As $\St_L(1)\in\cd(\SL_n(p^a))$, we have $p^{bm^2}=p^{an(n-1)/2}$
and hence \begin{equation}\label{3} 2bm^2=an(n-1).\end{equation}
Moreover, since $(p^{2bm}-1)\mid |L|\mid |\SL_n(p^a)|$, $2bm\leq an$
or $2bm=6>an$ by Zsigmondy's theorem. The case $2bm=6>an$ indeed
does not happen by \eqref{3} and the fact that $m\geq2$. Thus we
must have $2bm\leq an$. It follows by \eqref{3}  that $m\geq n-1$
and hence $2b(n-1)\leq an$. In particular, we obtain $b<a$ and
$a\geq2$.

From the description of unipotent characters of $|L|$ (cf.~\cite[p.
466]{C1}), we see that $L$ has a unipotent character $\chi$ of
degree $\chi(1)=(p^{bm}-1)(p^{bm}-p^b)/2(p^b+1)$ and moreover
$$\chi(1)<p^{2bm}\leq p^{an}.$$ We note that $\chi(1)$ is indeed the smallest degree of unipotent characters
of $L$ excluding some exceptions (cf.~\cite[Table~1]{Nguyen}).
Recall that $\chi(1)\in\cd(\SL_n(p^a))$ and we will get to a
contradiction by showing that $\chi(1)$ can not be equal to a degree
smaller than $p^{an}$ of $\SL_n(p^a)$.

When $n\geq5$, by inspecting the list of three smallest nontrivial
degrees of $\SL_n(q)$ in~\cite[Table~IV]{TZ1}, we see that
$$d_3(\SL_n(p^a))\geq\frac{(p^{an}-1)(p^{a(n-1)}-p^{2a})}{(p^a-1)(p^{2a}-1)}>p^{an}.$$
We deduce that $\chi(1)<d_3(\SL_n(p^a))$ and hence
$$\chi(1)=d_1(\SL_n(p^a))=\frac{p^{an}-p^a}{p^a-1}\text{ or }\chi(1)=d_2(\SL_n(p^a))=\frac{p^{an}-1}{p^a-1}.$$
In particular, $|\chi(1)|_p=1$ or $p^a$, which is impossible as
$\chi(1)=(p^{bm}-1)(p^{bm}-p^b)/2(p^b+1)$ and $b<a$.

For $n\leq4$, one can argue similarly by using the list of character
degrees of $\SL_n(q)$ available in the website of Lubeck
(cf.~\cite{Lubeck}). In fact, from there, one observes that every
degree smaller than $q^n$ of $\SL_n(q)$ ($n\leq4$) has $p$-part
either $1$, or $q$ (cf.~Table~\ref{table-linear}). Therefore,
$\chi(1)$ can not be one of these degrees.

$\bullet$ $L=P\Omega^\pm_{2m}(p^b)$ with $m\geq4$. Then we have
$p^{bm(m-1)}=p^{an(n-1)/2}$ and hence
\begin{equation}\label{4}2bm(m-1)=an(n-1).\end{equation} Moreover, since $(p^{2b(m-1)}-1)\mid
|L|\mid |\SL_n(p^a)|$, it follows by Zsigmondy's theorem that
$2b(m-1)\leq an.$ Therefore, $m\geq n-1$ and also $2b(n-2)\leq an$.
We claim that $b<a$. Assume the contrary, then $b\geq a$ and hence
\eqref{4} implies that $n(n-1)\geq 2m(m-1)$. In particular, $n\geq
5$. Now the inequality $2b(n-2)\leq an$ is violated.

From the description of unipotent characters of $|L|$ (cf.~\cite[p.
471]{C1}), we see that $L$ has a unipotent character $\chi$ of
degree $(p^{bm}\mp1)(p^{b(m-1)}\pm p^b)/(p^{2b}-1)$, where
$\chi(1)<p^{2b(m-1)}\leq p^{an}$. Now one just argues as in the
previous case.

$\bullet$ $L=\PSU_m(p^b)$ with $m\geq3$. As $|L|_p=|\SL_n(p^a)|_p$,
we have $bm(m-1)=an(n-1)$. Since $|L|\mid |\SL_n(p^a)|$, the
Zsigmondy's theorem implies that $b(m-1)\leq an$. It then follows
that $m\geq n-1$ and hence $an\geq b(n-2)$, which in turn implies
that $b\leq 2a$. Moreover, if $b=2a$ then we must have $n=4$ and
$m=3$ and it is easy to check that $\cd(L)\nsubseteq\cd(\SL_n(q))$
in this case. So we conclude that $b<2a$.

By~\cite[p. 465]{C1}, $\PSU_m(p^b)$ has a unipotent character $\chi$
of degree $(p^{bm}+(-1)^mp^b)/(p^b+1)$. Observe that
$\chi(1)<p^{b(m-1)}$ and therefore $\chi(1)<p^{an}$. As before,
since $|\chi(1)|_p=p^b$ and the $p$-part of a degree smaller than
$p^{an}$ of $\SL_n(p^a)$ is either $1$ or $p^a$
(cf.~Table~\ref{table-linear}), we deduce that $b=a$, which in turn
implies that $m=n$ as $bm(m-1)=an(n-1)$. This leads to a
contradiction as $\cd(L)=\cd(\PSU_n(p^a))\nsubseteq
\cd(\SL_n(p^a))$.

$\bullet$ $L=\PSL_m(p^b)$ with $m\geq2$. Arguing exactly as in the
unitary group case, we obtain that $a=b$ and $m=n$. This means
$L=\PSL_n(p^a)$, as wanted.
\end{proof}

%%%%%%%%%%%%%%%%%%%%%%%%%%%%%%%%%%%%%%%%%%%%%%%%%%%%%%%%%%%%%%%%%%%%%%%%%%%%%%%%%%%%%%%%%

\section{Unitary groups}\label{section-unitary groups}

If we want to prove Proposition~\ref{non-containment of character
degrees of quasisimple groups} for unitary groups similarly to the
case of linear groups, we need to know all the character degrees up
to $q^{2n}$ of $\SU_n(q)$. The characters of relatively small
degrees of unitary groups have been worked out
in~\cite{Libeck-O'Brian-Shalev-Tiep}. From there, one can obtain a
list of those degrees smaller than $q^{2n}$ when $n$ is large
enough, say $n\geq10$. However, when $n\leq9$, the list is fairly
long and therefore the arguments would be quite complicated. In this
section, we have found some new arguments to avoid the analysis of
degrees of $\SU_n(q)$ in small dimensions.

\begin{table}[h]\caption{Possible nontrivial character degrees up to $q^{2n}$ of $\SU_n(q)$ (cf.~\cite{Libeck-O'Brian-Shalev-Tiep}).}\label{table-unitary}
\begin{tabular}{ll}
Value of $n$ & Degrees\\\hline
 $n\geq10$& \begin{tabular}{l}$\frac{q^n+(-1)^nq}{q+1},
 \frac{(q^n-(-1)^n)(q^{n-1}+(-1)^nq^2)}{(q+1)(q^2-1)},
 \frac{(q^n+(-1)^nq)(q^n-(-1)^nq^2)}{(q+1)(q^2-1)},$\\$
  \frac{(q^n-(-1)^n)(q^n+(-1)^nq)}{(q+1)(q^2-1)},
   \frac{(q^n-(-1)^n)(q^{n-1}+(-1)^nq)}{(q+1)^2}$\end{tabular}\\\hline
\end{tabular}
\end{table}

\begin{lemma}\label{lemma for unitary groups} Proposition~\ref{non-containment of character degrees of quasisimple
groups} is true when $S$ is a simple unitary group.
\end{lemma}

\begin{proof} Since we can verify the lemma for the groups with exceptional Schur multiplier such as $\PSU_4(2), \PSU_4(3),\PSU_6(2)$
by using~\cite{Atl1}, we assume that $S=\PSU_n(q)=\PSU_n(p^a)$
($n\geq3$) is none of these. In that case, $\Schur(S)=\SU_n(p^a)$
and $\St_S(1)=p^{an(n-1)/2}$ is the only $p$-power degree of
$\SU_n(q)$ by~\cite[Theorem~1.1]{Malle-Zalesskii}. Recall that
$L:=G/M$ is a simple group of Lie type in characteristic $p$ and we
aim to show that $L\cong S$.

\medskip

(i) We first eliminate the simple Lie-type groups of exceptional
type. Suppose that $|L|_p=p^{bm}$ and the largest factor of the form
$p^x-1$ in the formula of $|L|$ is $p^{bm'}-1$. Then, as
$\St_L(1)=p^{bm}$ is a $p$-power degree of $\SU_n(q)$, we get
$p^{bm}=p^{an(n-1)/2}$ and therefore
\begin{equation}\label{bm=an(n-1)/2}bm=an(n-1)/2.\end{equation} It follows
that
$$bm'=\frac{an(n-1)}{2}\cdot \frac{m'}{m}.$$
As $$(p^{bm'}-1)\mid |\SU_n(q)|=p^{an(n-1)/2}\prod_{i=2}^n
(p^{ia}-(-1)^i),$$ we have $$(p^{bm'}-1)\mid \prod_{i=2}^n
(p^{2ia}-1).$$ Using Zsigmondy's theorem and recalling that
$n\geq3$, we deduce
$$bm'=\frac{an(n-1)}{2}\cdot \frac{m'}{m}\leq 2an.$$ Therefore,
$$n\leq 4m/m'+1.$$ Now we just argue as in the proof of Lemma~\ref{lemma for linear
groups} to get a contradiction.

\medskip

(ii) Next, we eliminate the simple classical groups in
characteristic $p$ except $\PSU_n(p^a)$.

$\bullet$ $L=\Omega_{2m+1}(p^b)$ or $\PSp_{2m}(p^b)$ with $m\geq2$.
As $\St_L(1)\in\cd(\SU_n(p^a))$, we have $p^{bm^2}=p^{an(n-1)/2}$
and hence \begin{equation}\label{5} 2bm^2=an(n-1).\end{equation}
This in particular implies that $b\neq 2a$ since $n(n-1)$ is never a
square. Moreover, since $(p^{2bm}-1)\mid |L|\mid |\SU_n(p^a)|$, it
follows by Zsigmondy's theorem that $bm\leq an$, which in turn
implies that $2m\geq n-1$ by~\eqref{5}. Therefore, $bm\leq an\leq
a(2m+1)<3am$ and hence $b<3a$.

Recall from the previous cases that $L$ has a unipotent character
$\chi$ of degree $\chi(1)=(p^{bm}-1)(p^{bm}-p^b)/2(p^b+1)$. As this
degree has $p$-part $p^b$ while the $p$-part of any degree of
$\SU_n(p^a)$ is a power of $p^a$, we deduce that $b=a$ since $2a\neq
b<3a$. Now equation~\ref{5} implies
 $$2m^2=n(n-1).$$ This happens only if $(n,m)=(9,6)$ or $n\geq 10$.
The former case leads to a contradiction since
$|\Omega_{13}(p^a)|=|\PSp_{12}(p^a)|$ does not divide
$|\SU_{9}(p^a)|$. Thus $n\geq 10$.

Recall that the degree of $\chi$ is
$(p^{bm}-1)(p^{bm}-p^b)/2(p^b+1)$, which is smaller than
$p^{2bm}\leq p^{2an}$. By inspecting the list of degrees smaller
than $q^{2n}$ of $\SU_n(q)$ with $n\geq 10$
in~\cite[Propositions~6.3 and~6.6]{Libeck-O'Brian-Shalev-Tiep},
which has been reproduced in Table~\ref{table-unitary}, we come up
with three possibilities as follow:
$$\chi(1)=\frac{q^n+(-1)^nq}{q+1}, \frac{(q^n-(-1)^n)(q^n+(-1)^nq)}{(q+1)(q^2-1)},
   \text{ or }\frac{(q^n-(-1)^n)(q^{n-1}+(-1)^nq)}{(q+1)^2},$$ where
   $q=p^a$. However, with the conditions $b=a$ and $2m^2=n(n-1)$, it is easy
to see that these equations have no solutions.

$\bullet$ $L=\PSL_m(p^b)$ with $m\geq2$. As $|L|_p=|\SU_n(p^a)|_p$,
we have $bm(m-1)=an(n-1)$. Since $|L|\mid |\SU_n(p^a)|$, the
Zsigmondy's theorem implies that $bm\leq 2an$. It then follows that
$2(m-1)\geq n-1$ and hence $bm\leq 2an\leq 2a(2m-1)<4am$. Thus
$b<4a$.

We know that $\PSL_m(p^b)$ has a unipotent character $\chi$ of
degree $(p^{bm}-p^b)/(p^b-1)$ and this degree belongs to
$\cd(\SU_n(p^a))$. As $|\chi(1)|_p=p^b$ and the $p$-part of any
degree of $\SU_n(p^a)$ is a power of $p^a$, the fact $b<4a$ implies
that $b=a,2a$, or $3a$. First, if $b=a$ then $m=n$ and this is
impossible since $|L|=|\PSL_n(p^a)|\nmid |\SU_n(p^a)| $ for every
$n\geq3$. Next, if $b=2a$ then
$$2m(m-1)=n(n-1),$$ which implies that $(n,m)=(4,3)$ or $n\geq10$.
The former case does not happen since
$|\PSL_m(p^b)|=|\PSL_3(p^{2a})|\nmid |\SU_4(p^a)|$. Hence we must
have $n\geq 10$. Recall that the degree of $\chi$ is
$(p^{bm}-p^b)/(p^b-1)$, which is smaller than $p^{2an}$. Using
Table~\ref{table-unitary} again, one sees that there is only one
possibility
$$\frac{p^{bm}-p^b}{p^b-1}=\frac{(p^{an}-(-1)^n)(p^{a(n-1)}+(-1)^np^{2a})}{(p^a+1)(p^{2a}-1)}.$$
Again, it is easy to see that this equation has no solutions since
$b=2a$ and $2m(m-1)=n(n-1)$,.

Finally, we consider the remaining case $b=3a$. Then
$3m(m-1)=n(n-1)$ and hence $(n,m)=(3,2)$ or $n\geq10$. The case
$(n,m)=(3,2)$ can not happen as $|L|=|\PSL_2(p^{3a})|>|\SU_3(p^a)|$
and the case $n\geq10$ is handled exactly as above.

$\bullet$ $L=P\Omega^\pm_{2m}(p^b)$ with $m\geq4$. Then we have
$p^{bm(m-1)}=p^{an(n-1)/2}$ and hence
\begin{equation}\label{6}2bm(m-1)=an(n-1).\end{equation} Moreover, since $(p^{2b(m-1)}-1)\mid
|L|\mid |\SU_n(p^a)|$, it follows by Zsigmondy's theorem that
$b(m-1)\leq an.$ Therefore, $2m\geq n-1$ and hence $b(m-1)\leq
a(2m+1)<4a(m-1)$. This means $b<4a$.

Recall $L$ has a unipotent character $\chi$ of degree
$(p^{bm}\mp1)(p^{b(m-1)}\pm p^b)/(p^{2b}-1)$, where
$\chi(1)<p^{2b(m-1)}\leq p^{2an}$. As above, we deduce that $b=a$,
$2a$, or $3a$.

First, if $b=3a$ then the inequality $b(m-1)\leq a(2m+1)$ implies
that $m=4$ and hence $n=9$. Now one can check that
$|L|=|P\Omega^\pm_8(p^{3a})|$ does not divide $|\SU_9(p^a)|$, a
contradiction. Second, if $b=2a$ then~\eqref{6} implies that
$$4m(m-1)=n(n-1).$$ Also, $$2(m-1)\leq n\leq 2m+1.$$ Now one see
that two above equations violate each other. Finally we assume that
$b=a$. Then it follows by~\eqref{6} that $$2m(m-1)=n(n-1),$$ which
again implies that $(n,m)=(4,3)$ or $n\geq 10$. Arguing similarly as
in the previous case, one gets to a contradiction.

$\bullet$ $L=\PSU_m(p^b)$ with $m\geq3$. Again we have
$bm(m-1)=an(n-1)$ and moreover $b(m-1)\leq 2an$ since $|L|\mid
|\SU_n(p^a)|$. We deduce that $2m\geq n-1$ and hence $b(m-1)\leq
2an\leq 2a(2m+1)\leq 7a(m-1)$, whence $b\leq7a$.

Recall that $\PSU_m(p^b)$ has a unipotent character $\chi$ of degree
$(p^{bm}+(-1)^mp^b)/(p^b+1)$. Now one just argues as in the previous
case to conclude that $b=a$, which also implies $m=n$. That means
$L=\PSU_n(p^a)$, as we wanted to prove.
\end{proof}

%%%%%%%%%%%%%%%%%%%%%%%%%%%%%%%%%%%%%%%%%%%%%%%%%%%%%%%%%%%%%%%%%%%%%%%%%%%%%%%%%%%%%%%

\section{Symplectic groups}\label{section-symplectic groups}

\begin{table}[h]\caption{Possible nontrivial degrees up to $q^{2n}$ of $\Sp_{2n}(q)$, $q$ odd (cf.~\cite{Lubeck,Nguyen}).}\label{table-symplectic}
\begin{tabular}{ll}
Value of $n$ & Degrees\\\hline
$n=2$ & \begin{tabular}{l}$\frac{q^2\pm1}{2},\frac{q(q\pm1)^2}{2},\frac{(q\pm1)(q^2+1)}{2},\frac{q(q^2+1)}{2},(q\pm1)(q^2+1),q(q^2+1),$\\
$\frac{(q^2+1)(q\pm1)^2}{2},\frac{q(q\pm1)(q^2+1)}{2},\frac{q^2(q^2\pm1)}{2},\frac{q^4-1}{2},(q^2+1)(q-1)^2,$\\$q(q-1)(q^2+1),(q^2-1)^2, q^4-1$\end{tabular}\\
 $n=3$ &\begin{tabular}{l}$\frac{q^3\pm1}{2},\frac{q(q\pm1)(q^3\pm1)}{2},\frac{(q^3\pm1)(q^2\pm q+1)}{2},\frac{q(q^2+1)(q^2\pm q+1)}{2}$\\$(q\pm1)(q^4+q^2+1),\frac{q(q\pm1)(q^4+q^2+1)}{2},
 \frac{q^6-1}{2},\frac{(q^2+1)(q^4+q^2+1)}{2},$\\$(q-1)(q^2+1)(q^3-1)$\end{tabular}\\

 $n\geq4$& \begin{tabular}{l}$\frac{q^n\pm 1}{2}, \frac{(q^n\pm 1)(q^n\pm q)}{2(q\pm 1)},\frac{q^{2n}-1}{2(q\pm 1)}, \frac{q^{2n}-1}{q\pm 1}$\end{tabular}\\\hline
\end{tabular}
\end{table}

\begin{table}[h]\caption{Possible nontrivial degrees up to $q^{2n}$ of $\Sp_{2n}(q)$, $q$ even (cf.~\cite{Guralnick-Tiep,Lubeck}).}\label{table-symplectic2}
\begin{tabular}{ll}
Value of $n$ & Degrees\\\hline
$n=2$ & \begin{tabular}{l}$\frac{q(q\pm1)^2}{2},\frac{q(q^2+1)}{2}, (q\pm1)(q^2+1),(q-1)^2(q^2+1),$\\$q(q-1)(q^2+1),(q^2-1)^2,q^4-1$\end{tabular}\\
 $n=3, q\neq 2$ &\begin{tabular}{l}$\frac{(q^3\pm 1)(q^3\pm q)}{2(q\pm 1)},\frac{q^{6}-1}{q\pm 1}, (q-1)(q^2+1)(q^3-1)$\end{tabular}\\

 $n\geq4, (n,q)\neq(4,2)$& \begin{tabular}{l}$\frac{(q^n\pm 1)(q^n\pm q)}{2(q\pm 1)},\frac{q^{2n}-1}{q\pm 1}$\end{tabular}\\\hline
\end{tabular}
\end{table}

\begin{lemma}\label{lemma for symplectic groups} Proposition~\ref{non-containment of character degrees of quasisimple
groups} is true when $S$ is a simple symplectic group.
\end{lemma}

\begin{proof} Assume that $S=\PSp_{2n}(q)=\PSp_{2n}(p^a)$ where $n\geq2$. As the case $\Sp_6(2)$
can be checked directly by~\cite{Atl1}, we assume that $(n,q)\neq(3,2)$ and hence $\Schur(S)=\Sp_{2n}(p^a)$ and
$\St_S(1)=p^{an^2}$ is the only $p$-power degree of $\Sp_{2n}(q)$
by~\cite[Theorem~1.1]{Malle-Zalesskii}. By Proposition~\ref{G/M is a
simple group of Lie type in characteristic p}, $L:=G/M$ is a simple
group of Lie type in characteristic $p$ and we aim to show that
$L\cong S$.

\medskip

(i) The simple Lie-type groups of exceptional type can be eliminated
as follows. Assume so and suppose that $|L|_p=p^{bm}$ and the
largest factor of the form $p^x-1$ in $|L|$ is $p^{bm'}-1$. Then, as
$\cd(L)\subseteq \cd(\Sp_{2n}(q))$, we get $p^{bm}=p^{an^2}$ and
therefore
\begin{equation}\label{bm=an^2}bm=an^2 \text{ and } bm'=an^2\cdot \frac{m'}{m}.\end{equation}
As $$(p^{bm'}-1)\mid |L|\mid |\Sp_{2n}(p^a)|=p^{an^2}\prod_{i=1}^n
(p^{2ia}-1),$$ using Zsigmondy's theorem, one gets
$$an^2\cdot \frac{m'}{m}\leq 2an \text{ or } an^2\cdot
\frac{m'}{m}=6>2an.$$ Since the latter case can not happen, we must
have $an^2m'/m\leq 2an$ and therefore $$n\leq 2m/m'.$$ Again, as
$m/m'$ is known, one obtains an upper bound for $n$. For small
values of $n$, we use Equation~\ref{bm=an^2} to show that either
$|L|\nmid |\SL_n(q)|$ or $\cd(L)\nsubseteq \cd(\SL_n(q))$ with the
help of~\cite{Lubeck} and Zsigmondy's theorem. Let us present the
arguments for the case $L(p^b)=E_7(p^b)$ as an example.

Assume by contrary that $L(p^b)=E_7(p^b)$. Then $m=63$ and $m'=18$.
Therefore $n\leq 7$. For every $n\leq 7$, we obtain a contradiction
by showing that $|E_7(p^b)|\nmid |\Sp_{2n}(p^a)|$. For instance, if
$n=7$ then $63b=49a$ by~\eqref{bm=an^2}. Therefore $9b=7a$ and hence
$a=9c$ and $b=7c$ for some positive integer $c$. Now the Zsigmondy
prime $\ell(p,1,98c)$ is a divisor of $|E_7(p^{7c})|$ but does not
divide $|\Sp_{14}(p^{9c})|$.

\medskip

(ii) We next eliminate the remaining simple classical groups in
characteristic $p$ except $\PSp_{2n}(p^a)$.

$\bullet$ $L=P\Omega^\pm_{2m}(p^b)$ with $m\geq4$. Then we have
$p^{bm(m-1)}=p^{an^2}$ and hence
\begin{equation}\label{1}bm(m-1)=an^2.\end{equation} Since $(p^{2b(m-1)}-1)\mid
|\Sp_{2n}(p^a)|$, it follows by Zsigmondy's theorem that $b(m-1)\leq
an$, whence $m\geq n$ by~\eqref{1}. This in particular implies that
$b<2a$. Moreover, as $m(m-1)$ can not be equal to $n^2$, we obtain
that $b\neq a$.

As before, $L$ has a unipotent character $\chi$ of degree
$(p^{bm}\mp1)(p^{b(m-1)}\pm p^b)/(p^{2b}-1)$ and
$$\chi(1)<p^{2b(m-1)}\leq p^{2an}.$$
First we consider the case when $q$ is odd. Using the classification
of low-dimensional irreducible characters of $\Sp_{2n}(q)$ of degree
up to $q^{2n}$ in~\cite[Corollary~4.2]{Nguyen} for $n\geq6$
and~\cite{Lubeck} for $n\leq5$, we see that all degrees less than
$p^{2an}$ of $\Sp_{2n}(p^a)$ have $p$-parts $1, p^a$, or $p^{2a}$.
Therefore, these degrees can not be $\chi(1)$ since
$|\chi(1)|_p=p^b$ where $b\neq a,2a$. We have shown that
$\chi(1)\notin\cd(Sp_{2n}(p^a))$, a contradiction. The case when $q$
even is handled exactly in the same way by using a result on
low-dimensional characters of symplectic groups in even
characteristic of Guralnick and
Tiep~\cite[Theorem~6.1]{Guralnick-Tiep}.

$\bullet$ $L=\PSL_m(p^b)$ or $\PSU_m(p^b)$ with $m\geq2$. Then we
have $p^{bm(m-1)/2}=p^{an^2}$ and hence $bm(m-1)=2an^2$. Since
$(p^{bm}-1)\mid |\Sp_{2n}(p^a)|$, it follows by Zsigmondy's theorem
that $bm\leq 2an$ since the case $6=bm>2an$ can not happen. It
follows that $m-1\geq n$ and hence $2an\geq b(n+1)$. In particular,
$b<2a$.

We already know that $L$ has a unipotent character $\chi$ of degree
$(p^{bm}\pm p^b)/(p^{b}\pm1)$, which is smaller than $p^{bm}\leq
p^{2an}.$ Using the list of characters of degrees up to $p^{2an}$ in
Tables~\ref{table-symplectic} and~\ref{table-symplectic2}, it is
routine to check that $\chi(1)$ can not be equal to any of such
degree. Indeed, as $|\chi(1)|_p=p^b$ and $b<2a$, one deduces that
$b=a$ and only needs to check for degrees with $p$-part $p^a$ or
$p^a/2$ when $p=2$.

$\bullet$ $L=\PSp_{2m}(p^b)$ or $\Omega_{2m+1}(p^b)$ with $m\geq2$.
Then we have $p^{bm^2}=p^{an^2}$ and hence $bm^2=an^2$ and in
particular we have $b\neq 2a$. Since $(p^{2bm}-1)\mid
|\Sp_{2n}(p^a)|$, it follows by Zsigmondy's theorem that $bm\leq an$
since the case $6=2bm>2an$ can not happen.

Recall that $L$ has a unipotent character $\chi$ of degree
$(p^{bm}-1)(p^{bm}-p^b)/2(p^{b}+1)$, which is smaller than
$p^{2bm}\leq p^{2an}.$ Arguing as in $L=P\Omega^\pm_{2m}(p^b)$ case
and recalling that $b\neq2a$, we obtain that $b=a$ and therefore
$m=n$. We now just need to eliminate the possibility
$L=\Omega_{2n+1}(p^a)$ with $n\geq3$ and $p\geq3$ (note that
$\Omega_5(q)\cong \PSp_4(q)$ and
$\Omega_{2n+1}(2^a)\cong\PSp_{2n}(2^a)$). As mentioned
in~\cite[\S6]{TZ1}, $\Omega_{2n+1}(p^a)$ with $n\geq3$ and $p$ an
odd prime has an irreducible character of degree
$(p^{2an}-1)/(p^{2a}-1)$, which is also smaller than $p^{2an}$.
Again, the classification of irreducible characters of
$\Sp_{2n}(p^a)$ of degrees smaller than $p^{2an}$ shows that
$\Sp_{2n}(p^a)$ has no such degree, as desired.
\end{proof}

%%%%%%%%%%%%%%%%%%%%%%%%%%%%%%%%%%%%%%%%%%%%%%%%%%%%%%%%%%%%%%%%%%%%%%%

\section{Orthogonal groups}\label{section-orthogonal groups in even
dimension}

\begin{table}[h]\caption{Possible nontrivial degrees up to $q^{2n}$ of $\Spin_{2n+1}(q)$, $n\geq3$, $q$ odd (cf.~\cite{Lubeck,Nguyen}).}\label{table-orthogonal in odd dimension}
\begin{tabular}{ll}
Value of $n$ & Degrees\\\hline
 $n=3$ &\begin{tabular}{l}$\frac{q^{6}-1}{q^2-1},\frac{q(q^{6}-1)}{q^2-1},\frac{(q^3\pm1)(q^3\pm q)}{2(q\pm1)},\frac{q^{6}-1}{q\pm1},$\\
 $\frac{1}{2}(q\pm1)(q^2+1)(q^3\pm1),(q-1)(q^2+1)(q^3-1)$\end{tabular}\\
 $n=4$& \begin{tabular}{l}$\frac{q^{8}-1}{q^2-1},\frac{q(q^{8}-1)}{q^2-1},\frac{(q^4\pm1)(q^4\pm q)}{2(q\pm1)},\frac{q^{8}-1}{q\pm1},$\\
 $\frac{1}{2}(q^2-1)(q^6-1),\frac{1}{2}(q^4+1)(q^4+q^2+1)$\end{tabular}\\

 $n\geq5$& \begin{tabular}{l}$\frac{q^{2n}-1}{q^2-1},\frac{q(q^{2n}-1)}{q^2-1},\frac{(q^n\pm1)(q^n\pm q)}{2(q\pm1)},\frac{q^{2n}-1}{q\pm1}$\end{tabular}\\\hline
\end{tabular}
\end{table}

\begin{lemma}\label{lemma for orthogonal groups in odd dimension} Proposition~\ref{non-containment of character degrees of quasisimple
groups} is true when $S$ is a simple orthogonal group in odd
dimension.
\end{lemma}

\begin{proof} Assume that $S=\Omega_{2n+1}(q)=\Omega_{2n+1}(p^a)$ where $n\geq3$ and $p$ an odd prime. As the case $\Omega_7(3)$ can be checked directly by~\cite{Atl1},
 we assume that $(n,q)\neq(3,3)$ and hence $\Schur(S)=\Spin_{2n+1}(p^a)$ and
$\St_S(1)=p^{an^2}$ is the only $p$-power degree of
$\Spin_{2n+1}(q)$ by~\cite[Theorem~1.1]{Malle-Zalesskii}. In view of
Proposition~\ref{G/M is a simple group of Lie type in characteristic
p}, we have known that $L:=G/M$ is a simple group of Lie type in
characteristic $p$ and we want to show that $L=S$. The simple groups
of exceptional Lie type is eliminated as in Lemma~\ref{lemma for
symplectic groups}. For classical groups, we also follow the proof
of Lemma~\ref{lemma for symplectic groups} and make use of the
classification of irreducible characters of low degrees of
$\Spin_{2n+1}(q)$ by Nguyen in \cite[Theorem~1.2]{Nguyen}.
\end{proof}

\begin{table}[h]\caption{Possible nontrivial degrees up to $q^{2n-2}$ of $\Spin^\pm_{2n}(q)$, $n\geq4$ (cf.~\cite{Lubeck,Nguyen}).}\label{table-orthogonal in even dimension}
\begin{tabular}{ll}
Value of $n,q$ & Degrees\\\hline

\begin{tabular}{l} $n\geq4$, $q$ odd\end{tabular} & $\frac{(q^n\pm1)(q^{n-1}\mp q)}{q^2-1},\frac{(q^n\pm1)(q^{n-1}\pm1)}{2(q\pm1)},\frac{(q^n\pm1)(q^{n-1}\pm1)}{q+1}$\\
 \begin{tabular}{l} $n\geq4$, $q$ even,\\ $(n,q)\neq(4,2),(5,2)$\end{tabular} & $\frac{(q^n\pm1)(q^{n-1}\mp q)}{q^2-1},\frac{(q^n\pm1)(q^{n-1}\pm1)}{q+1}$\\\hline
\end{tabular}
\end{table}

\begin{lemma}\label{lemma for orthogonal groups in even dimension} Proposition~\ref{non-containment of character degrees of quasisimple
groups} is true when $S$ is a simple orthogonal group in even
dimension.
\end{lemma}

\begin{proof} We only present here the proof of the case
$S=P\Omega_{2n}^+(p^a)$ with $n\geq4$. The minus type orthogonal
groups are dealt similarly.

We assume that $S\neq P\Omega_8^+(2)$ as this special case can be
eliminated easily by using~\cite{Atl1}. Then
$\Schur(S)=\Spin_{2n}^+(p^a)$ and $\St_S(1)=p^{an(n-1)}$ is the only
$p$-power degree of $\Spin_{2n}(p^a)$
by~\cite[Theorem~1.1]{Malle-Zalesskii}. We have already known that
$L:=G/M$ is a simple group of Lie type in characteristic $p$ and we
aim to show that $L=S$.

\medskip

(i) The simple Lie-type groups of exceptional type can be eliminated
as follows. Assume that $L$ is of exceptional type. Suppose that
$|L|_p=p^{bm}$ and the largest divisor of $|L|$ of the form $p^x-1$
is $p^{bm'}-1$. Then, as $\cd(L)\subseteq \cd(\Spin^+_{2n}(p^a))$,
we get $p^{bm}=p^{an(n-1)}$ and therefore
\begin{equation}\label{bm=an(n-1)}bm=an(n-1) \text{ and } bm'=an(n-1)\cdot \frac{m'}{m}.\end{equation}
As $$(p^{bm'}-1)\mid
|\Spin^+_{2n}(p^a)|=p^{an(n-1)}(p^{an}-1)\prod_{i=1}^{n-1}
(p^{2ia}-1),$$ using Zsigmondy's theorem and note that $n\geq4$, one
gets $an(n-1)\cdot \frac{m'}{m}\leq 2a(n-1),$ and therefore
$$n\leq 2m/m'.$$ Now one can argue as in Lemma~\ref{lemma for symplectic
groups} to get a contradiction.

\medskip

(ii) Next, we eliminate the remaining simple classical groups in
characteristic $p$ except $P\Omega_{2n}^+(p^a)$.

$\bullet$ $L=\PSp_{2m}(p^b)$ or $\Omega_{2m+1}(p^b)$ with $m\geq2$.
Then we have $p^{bm^2}=p^{an(n-1)}$ and hence $bm^2=an(n-1)$. Since
$(p^{2bm}-1)\mid |\Spin^+_{2n}(p^a)|$, it follows by Zsigmondy's
theorem that $bm\leq a(n-1)$ as $n\geq4$. We then obtain that $m\geq
n$, whence $a(n-1)\geq bm\geq bn$. In particular, $b<a$.

Recall that $L$ has a unipotent character $\chi$ of degree
$(p^{bm}-1)(p^{bm}-p^b)/2(p^{b}+1)$, which is smaller than
$p^{2bm}\leq p^{2a(n-1)}.$ Using the classification of irreducible
characters of $\Spin_{2n}^{\pm}(p^a)$ of degree smaller than
$q^{2(n-1)}$ in~\cite[Theorems~1.3,1.4]{Nguyen} for $n\geq5$ and the
list of degrees of $\Spin_{2n}^{\pm}(p^a)$ for $n\leq4$
in~\cite{Lubeck}, we observe that the $p$-part of such a degrees is
either nothing or $p^a$ (cf.~Table~\ref{table-orthogonal in even
dimension}) and hence any of them can not be equal to $\chi(1)$ as
$|\chi(1)|_p=p^b$ and $a>b$.

$\bullet$ $L=\PSL_m(p^b)$ with $m\geq2$. Then we have
$p^{bm(m-1)/2}=p^{an(n-1)}$ and hence $bm(m-1)=2an(n-1)$. Since
$(p^{bm}-1)\mid |\Spin^+_{2n}(p^a)|$, we deduce by Zsigmondy's
theorem that $bm\leq 2a(n-1)$ and hence $m-1\geq n$. Thus
$2a(n-1)\geq bm\geq b(n+1)$, which in turn implies that $b<2a$.

We know that $L$ has a unipotent character $\chi$ of degree
$(p^{bm}- p^b)/(p^{b}-1)$, which is smaller than $p^{bm}\leq
p^{2a(n-1)}.$ Now using the results
in~\cite[Theorems~1.3,1.4]{Nguyen} and~\cite{Lubeck} again, we see
that every degree of $\Spin_{2n}^+(p^a)$ smaller than $p^{2a(n-1)}$
has $p$-part $1,p^a$ or $p^{2a}$. As, $|\chi(1)|=p^b$ and $ b<2a$,
it follows that $b=a$ and therefore we have
$$m(m-1)=2n(n-1).$$ In particular, we obtain $n\geq 6$. Table~\ref{table-orthogonal in even
dimension} now says that $\Spin^\pm_{2n}(p^a)$ has the only two
following degrees that are smaller than $p^2a(n-2)$ and divisible by
$p$:
$$\frac{(p^{an}\pm1)(p^{a(n-1)}\mp p^a)}{p^{2a}-1}.$$ However, one
can check that $$\chi(1)=\frac{p^{bm}- p^b}{p^{b}-1}=\frac{p^{am}-
p^a}{p^{a}-1}$$ is not equal to neither of them as $m(m-1)=2n(n-1)$.

$\bullet$ $L=P\Omega^\pm_{2m}(p^b)$ with $m\geq4$. Then we have
$p^{bm(m-1)}=p^{an(n-1)}$ and hence
\begin{equation}\label{2}bm(m-1)=an(n-1).\end{equation} Since $(p^{2b(m-1)}-1)\mid
|\Spin^+_{2n}(p^a)|$, it follows by Zsigmondy's theorem that
$b(m-1)\leq a(n-1)$, whence $m\geq n$ by~\eqref{2}. Therefore,
$a(n-1)\geq b(m-1)\geq b(n-1)$ and so $b<2a$.

As before, $L$ has a unipotent character $\chi$ of degree
$(p^{bm}\mp1)(p^{b(m-1)}\pm p^b)/(p^{2b}-1)$ and
$$\chi(1)<p^{2b(m-1)}\leq p^{2a(n-1)}.$$ As in the previous case, by using the results
in~\cite[Theorems~1.3,1.4]{Nguyen} and~\cite{Lubeck}, we obtain that
$\chi(1)\in\cd(\Spin_{2n}^+(p^a))$ only if $a=b$, which also means
$m=n$.

We now just need to eliminate $L=P\Omega^-_{2n}(p^a)$. In fact, even
in this case, one can check that the degree
$(p^{an}+1)(p^{a(n-1)}-p^a)/(p^{2a}-1)$ of $P\Omega^-_{2n}(p^a)$ is
not equal to any degree of $\Spin_{2n}^+(p^n)$ of degree smaller
than $p^{2a(n-1)}$.
\end{proof}

%%%%%%%%%%%%%%%%%%%%%%%%%%%%%%%%%%%%%%%%%%%%%%%%%%%%%%%%%%%%%%%%%%%%%%%%%%%%%%%%%%%%%%%%%%%

\section{Proof of the main result}\label{section-main result}

\begin{proof}[Proof of Proposition~\ref{non-containment of character degrees of quasisimple
groups}] This is an immediate consequence of Lemmas~\ref{lemma for
linear groups}, \ref{lemma for unitary groups}, \ref{lemma for
symplectic groups}, \ref{lemma for orthogonal groups in odd
dimension}, and~\ref{lemma for orthogonal groups in even dimension}.
\end{proof}

We establish a couple of important lemmas leading to the proof of
the main theorem at the end of the section. The next two lemmas are
well known. The first one is due to Bianchi, Chillag, Lewis, and
Pacifici and the second one is due to Moret\'{o}.

\begin{lemma}[\cite{Bianchi}, Lemma 5]\label{lemmaBianchi} Let $N=T\times\cdot\cdot\cdot\times T$, a direct product
of $k$ copies of a nonabelian simple group $T$, be a minimal normal
subgroup of $G$. If $\chi\in\Irr(T)$ extends to $\Aut(T)$, then
$\chi^k$ extends to a character of $G$.
\end{lemma}

\begin{lemma}[\cite{Moreto1}, Lemma~4.2]\label{lemma Moreto} Let $T$
be a nonabelian simple group. Then there exists a non-principal
irreducible character of $T$ that extends to $\Aut(T)$.
\end{lemma}

The following lemma basically says that every nonabelian simple
group different from $\PSL_3(4)$ and $\PSU_4(3)$ can not be embedded
in to an automorphism group of an abelian group of order not bigger
than than of $\Mult(S)$.

\begin{lemma}\label{S cannot be embedded in Aut(A)} Let $S$ be a
nonabelian simple group different from $\PSL_3(4)$ and $\PSU_4(3)$.
Let $A$ be an abelian group of order less than or equal to
$|\Mult(S)|$. Then $|S|>|\Aut(A)|$.
\end{lemma}

\begin{proof} If $|\Mult(S)|\leq4$ then the lemma is obvious. Therefore, one only needs
to consider nonabelian simple groups with Schur multiplier of order
at least $5$.

First we consider $S=\PSL_2(9)$ or $S=\Omega_7(3)$. Then
$|\Mult(S)|=6$. If $|A|\leq 5$ then $|\Aut(A)|\leq 5!=120$ and we
are done. On the other hand, if $|A|=6$ then $A\cong \ZZ_6$ as $A$
is abelian and we are done also. Next we consider $S=\PSU_6(2)$ or
$\ta E_6(2)$. Then $|\Mult(S)|=12$ and therefore
$|S|>12!>|\Aut(A)|$, as desired.

Now we consider the linear groups $S=\PSL_n(q)$. Excluding all the
exceptional cases already considered, we have
$|\Mult(S)|=(n,q-1)\leq n$ and so $|\Aut(A)|\leq n!$. On the other
hand, it is easy to check that $|S|>n!$. For the simple unitary
groups $S=\PSU_n(q)$, we still have $|\Mult(S)|=(n,q+1)\leq n$ and
$|S|>n!$, which imply the lemma.
\end{proof}

Recall that, for each nonnegative integer $i$, $M^{(i)}$ denotes the
$i$th derived subgroup of $M$.

\begin{lemma}\label{G/Mi is isomorphic to a quotient of Schur(S)} Let $S$ be a nonabelian simple group different from $\PSL_3(4)$ and $\PSU_4(3)$.
Let $G$ be a perfect group and $M\lhd G$ such that $G/M\cong S$ and
$|M|\leq|\Mult(S)|$. Then, for every nonnegative integer $i$,
$G/M^{(i)}$ is isomorphic to a quotient of $\Schur(S)$.
\end{lemma}

\begin{proof} We prove by induction that $G/M^{(i)}$ is
isomorphic to a quotient of $\Schur(S)$ for every $i$. The induction
base $i=0$ is exactly the hypothesis. Now assuming that
$G/M^{(i)}\cong \Schur(S)/Z_i$ for some normal subgroup $Z_i$ of
$\Schur(S)$, we have to show $G/M^{(i+1)}$ is also a quotient of
$\Schur(S)$.

As $M^{(i)}/M^{(i+1)}$ is abelian and normal in $G/M^{(i+1)}$, we
have $$\frac{M^{(i)}}{M^{(i+1)}}\leq
C_{G/M^{(i+1)}}(\frac{M^{(i)}}{M^{(i+1)}})\unlhd
\frac{G}{M^{(i+1)}}.$$ We first consider the case
$C_{G/M^{(i+1)}}(M^{(i)}/M^{(i+1)})=G/M^{(i+1)}$. Then
$M^{(i)}/M^{(i+1)}$ is central in $G/M^{(i+1)}$. As $G$ is perfect,
$G/M^{(i+1)}$ is a stem extension of $G/M^{(i)}\cong \Schur(S)/Z_i$.
As $\Schur(S)/Z_i$ is a quasisimple group whose quotient by the
center is $S$, we deduce from the main result of~\cite{Harris} that
$G/M^{(i+1)}$ is a quotient of the Schur cover of $\Schur(S)/Z_i$.
Therefore, $G/M^{(i+1)}$ is a quotient of $\Schur(S)$, as wanted.

The lemma is completely proved if we can show that
$C_{G/M^{(i+1)}}(M^{(i)}/M^{(i+1)})$ can not be a \emph{proper}
normal subgroup of $G/M^{(i+1)}$. Assume so, then it follows by the
induction hypothesis that
$$\frac{C_{G/M^{(i+1)}}(M^{(i)}/M^{(i+1)})}{M^{(i)}/M^{(i+1)}}\lhd
\frac{G/M^{(i+1)}}{M^{(i)}/M^{(i+1)}}\cong
\frac{G}{M^{(i)}}=\frac{\Schur(S)}{Z_i}.$$ Therefore,
$$\Big{|}\frac{C_{G/M^{(i+1)}}(M^{(i)}/M^{(i+1)})}{M^{(i)}/M^{(i+1)}}\Big|\leq
\Big|\frac{\Mult(S)}{Z_i}\Big|=\Big|\frac{M}{M^{(i)}}\Big|$$ and
hence
$$|C_{G/M^{(i+1)}}(M^{(i)}/M^{(i+1)})|\leq |M/M^{(i+1)}|.$$ Thus
$$\Big|\frac{G/M^{(i+1)}}{C_{G/M^{(i+1)}}(M^{(i)}/M^{(i+1)})}\Big|\geq
|G/M|=|S|.$$ Since the quotient group on the left side can be
embedded in $\Aut(M^{(i)}/M^{(i+1)})$ and $M^{(i)}/M^{(i+1)}$ is
abelian of order less than or equal to $|M|$, this last inequality
leads to a contradiction by Lemma~\ref{S cannot be embedded in
Aut(A)}.
\end{proof}

\begin{lemma}\label{G is isomorphic to a quotient of Schur(S)} Let $S$ be a simple group of Lie type different from $\PSL_3(4)$ and $\PSU_4(3)$.
Let $G$ be a perfect group and $M\lhd G$ such that $G/M\cong S$,
$|M|\leq|\Mult(S)|$, and $\cd(G)\subseteq\cd(\Schur(S))$. Then $G$
is isomorphic to a quotient of $\Schur(S)$.
\end{lemma}

\begin{proof} By Lemma~\ref{G/Mi is isomorphic to a quotient of
Schur(S)}, we are done if $M$ is solvable. So it remains to consider
the case when $M$ is nonsolvable. If $M$ is nonsolvable, there is an
integer $i$ such that $$M^{(i)}=M^{(i+1)}> 1.$$ Let $N\leq M^{(i)}$
be a normal subgroup of $G$ so that $M^{(i)}/N\cong T^k$ for some
non-abelian simple group $T$. By Lemma~\ref{lemma Moreto}, $T$ has a
non-principal irreducible character $\varphi$ that extends to
$\Aut(T)$. Lemma~\ref{lemmaBianchi} then implies that $\varphi^k$
extends to $G/N$. Therefore, by Gallagher's lemma,
$\varphi^k\chi\in\Irr(G/N)$ for every $\chi\in\Irr(G/M^{(i)})$. In
particular,
$$\varphi(1)^k\chi(1)\in\cd(G/N)\subseteq\cd(G)\subseteq\cd(\Schur(S)).$$
Taking $\chi$ to be the Steinberg character of $S$. By
Lemma~\ref{G/Mi is isomorphic to a quotient of Schur(S)}, $S$ is a
quotient of $G/M^{(i)}$ and hence $\chi$ can be considered as a
character of of $G/M^{(i)}$. We now get a contradiction since
$\varphi(1)^k\chi(1)$, which is larger than $\chi(1)=\St_S(1)$, can
not be degree of $\Schur(S)$ by Lemma~\ref{degree of S whose no
proper multiple is degree of Schur(S)}.
\end{proof}

\begin{lemma}\label{G is uniquely determined by S and |G|} Assume the hypothesis of Lemma~\ref{G is isomorphic to a quotient of
Schur(S)}. Then $G$ is uniquely determined (up to isomorphism) by
$S$ and the order of $G$.
\end{lemma}

\begin{remark} In the case $S=\PSU_4(3)$, this lemma is true if Lemma~\ref{S cannot be embedded in Aut(A)} is true. However, this lemma is
wrong in the case $S=\PSL_3(4)$ because of the following. Note that
$\Mult(S)=\ZZ_4\times\ZZ_4\times\ZZ_3$. Let $Z_1$ and $Z_2$ be
subgroups of $\Mult(S)$ isomorphic respectively to $\ZZ_4$ and
$\ZZ_2\times\ZZ_2$. The non-isomorphic groups $\Schur(S)/Z_1$ and
$\Schur(S)/Z_2$ (cf.~\cite{Atl1} where these groups are denoted by
$12_1.S$ and $12_2.S$) both satisfies the hypothesis of the lemma.
\end{remark}

\begin{proof} First we consider the case $S=P\Omega_8^+(2), Suz(8)$, or
$P\Omega_{2n}^+(q)$ with $n$ even and $q$ odd. Then $\Mult(S)\cong
\ZZ_2\times\ZZ_2$. By Lemma~\ref{G is isomorphic to a quotient of
Schur(S)}, $S$ is isomorphic to a quotient of $\Schur(S)$ so that we
can assume $G\cong \Schur(S)/Z$ with $Z\leq \Mult(S)$ (note that $Z$
can not be $\Schur(S)$). If $|M|=1$ or $4$ then $Z=\Mult(S)$ or $1$,
respectively, and so we are done. Thus it remains to consider
$|M|=2$. We then have $|Z|=2$ and hence $Z$ is generated by an
involution of $\Mult(S)$. However, as these three involutions are
permuted by an outer automorphism of $S$ of degree $3$, the quotient
groups of the form $\Schur(S)/\langle t\rangle$ for any involution
$t\in\Mult(S)$ are isomorphic and we are done again. Next, we assume
that $S=\PSU_6(2)$ or $\ta E_6(2)$. Though the Schur multipliers of
these groups are more complicated, these cases in fact can be argued
similarly as above. We leave details to the reader.

If $S$ is none of the groups already considered and also
$S\neq\PSL_3(4)$ and $\PSU_4(3)$, then $\Mult(S)$ indeed is cyclic.
Again, $S$ is isomorphic to a quotient of $\Schur(S)$ and we can
assume
$$G\cong \Schur(S)/Z,$$ where $Z\leq \Mult(S)$ . As $G/M\cong S$, we
then deduce that $|Z|=|\Mult(S)|/|M|=|\Schur(S)|/|G|$. Since the
cyclic group $\Mult(S)$ has a unique subgroup of order
$|\Schur(S)|/|G|$, $Z$ is uniquely determined by $S$ and $|G|$ and
therefore the lemma follows.
\end{proof}

We are now ready to prove the main result.

\begin{proof}[Proof of Theorem~\ref{main theorem}]. Recall the
hypothesis that $\CC G\cong\CC H$ where $H$ is a quasisimple group
such that $S:=H/Z(H)$ is a simple classical group different from
$\PSL_3(4)$ and $\PSU_4(3)$. As $H$ has only one linear character,
we deduce that so does $G$ and hence $G$ is perfect. It is obvious
that $G$ is nontrivial. Therefore, if $M$ is a maximal normal
subgroup of $G$, then $G/M$ is isomorphic to a nonabelian simple
group. Moreover, as $\CC G\cong \CC H$,
$$\cd(G/M)\subseteq \cd(G)=\cd(H)\subseteq \cd(\Schur(S)).$$
Applying Proposition~\ref{non-containment of character degrees of
quasisimple groups}, we obtain that $$G/M\cong S.$$ It follows that
$|M|=|Z(H)|\leq |\Mult(S)|$ since $|G|=|H|$. Now we see that the two
groups of the same order $G$ and $H$ both satisfy the hypothesis of
Lemma~\ref{G is isomorphic to a quotient of Schur(S)}. It then
follows by Lemma~\ref{G is uniquely determined by S and |G|} that
$G\cong H$, as desired.
\end{proof}

%%%%%%%%%%%%%%%%%%%%%%%%%%%%%%%%%%%%%%%%%%%%%%%%%%%%%%%%%%%%%%%%%%%%%%%%%%%%%%%%%%%%%%%%%%%%%%%%%%%%%%%%%%%%%%%%%%%%

\section*{Acknowledgement} The author is grateful to Hung P. Tong-Viet for many fruitful discussions on the topic. He also acknowledges the support of a
Faculty Research Grant (FRG 1747) from The University of Akron.

%%% -------------------------------------------------------------------------------------------------------

%\bibliographystyle{amsplain}

\end{document}